\documentclass[12pt]{article}
\usepackage{amsthm,wasysym,amssymb,eufrak,indentfirst,cite,xcolor}
\begin{document}
\baselineskip 17pt
\newtheorem{Theorem}{Theorem}[section]
\newtheorem{Lemma}[Theorem]{Lemma}
\newtheorem{Proposition}[Theorem]{Proposition}
\newtheorem{Example}[Theorem]{Example}
\newtheorem{Definition}[Theorem]{Definition}
\newtheorem{Corollary}[Theorem]{Corollary}
\newtheorem{Remark}[Theorem]{Remark}
\newtheorem{Question}[Theorem]{Question}

\title{On finite groups with some primary subgroups satisfying partial $S$-$\Pi$-property\thanks{Research is supported by an NNSF of China (grant No. 11371335) and Wu Wen-Tsuu Key Laboratory of Mathematics of Chinese Academy of Sciences. The first author is also supported by the Start-up Scientific Research Foundation of Nanjing Normal University (grant No. 2015101XGQ0105) and a project funded by the Priority Academic Program Development of Jiangsu Higher Education Institutions.}}

\author{{Xiaoyu Chen$^{1}$, Yuemei Mao$^{2,3}$, Wenbin Guo$^{2,}$\thanks{Corresponding author.}}\\
{\small $^1$School of Mathematical Sciences and Institute of Mathematics, Nanjing Normal University,}\\
{\small Nanjing 210023, P. R. China}\\
{\small $^2$School of Mathematical Sciences, University of Science and Technology of China,}\\
{\small Hefei 230026, P. R. China}\\
{\small $^3$School of Mathematics and Computer Science, University of Datong of Shanxi,}\\
{\small Datong 037009, P. R. China}\\
{\small E-mails: jelly@njnu.edu.cn, maoym@mail.ustc.edu.cn, wbguo@ustc.edu.cn}}
\date{}
\maketitle

\begin{abstract}  A $p$-subgroup $H$ of a finite group $G$ is said to satisfy partial $S$-$\Pi$-property in $G$ if $G$ has a chief series $\Gamma_{G}: 1=G_{0}<G_{1}<\cdots<G_{n}=G$ such that for every $G$-chief factor $G_{i}/G_{i-1}$ $(1\leqslant i\leqslant n)$ of $\Gamma_{G}$, either $(H\cap G_{i})G_{i-1}/G_{i-1}$ is a Sylow $p$-subgroup of $G_{i}/G_{i-1}$ or $|G/G_{i-1}: N_{G/G_{i-1}}((H\cap G_{i})G_{i-1}/G_{i-1})|$ is a $p$-number. In this paper, we mainly investigate the structure of finite groups with some primary subgroups satisfying partial $S$-$\Pi$-property.
\end{abstract}

\let\thefootnoteorig\thefootnote
\renewcommand{\thefootnote}{\empty}


 \let\thefootnote\thefootnoteorig

\section{Introduction}

Throughout this paper, all groups considered are finite. $G$ always denotes a group, $p$ denotes a prime, and $|G|_{p}$ denotes the order of Sylow $p$-subgroups of $G$. Also, we use $\mathfrak{U}$ and $\mathfrak{N}$ to denote the classes of all supersoluble groups and nilpotent groups, respectively. \par

Recall that a subgroup $H$ of $G$ has the cover-avoidance property in $G$ or $H$ is called a $CAP$-subgroup of $G$ if either $H$ covers $L/K$ (i.e. $L\leq HK$) or $H$ avoids $L/K$ (i.e.  $H\cap L\leq K$) for each $G$-chief factor $L/K$. Also, a subgroup $H$ of $G$ is  said to be $S$-quasinormally embedded \cite{BalP} in $G$ if each Sylow subgroup of $H$ is also a Sylow subgroup of some $S$-quasinormal subgroup of $G$. The $CAP$-subgroups and $S$-quasinormally embedded subgroups play an important role in the study of the structure of finite groups, and have been investigated by many authors.  As a generalization of $CAP$-subgroup and $S$-quasinormally embedded subgroup, W. Guo, A. N. Skiba and N. Yang introduced the concept of
generalized $CAP$-subgroup \cite{GSY}: a subgroup $H$ of  $G$ is said to be a generalized $CAP$-subgroup of $G$  if for each $G$-chief factor $L/K$,
either $H$  avoids $L/K$ or the following hold: (1) If $L/K$ is non-abelian, then $|L:(H\cap L)K|$ is a $p'$-number for every $p\in \pi ((H\cap L)K/K)$;
(2)  If $L/K$ is a $p$-group, then $|G:N_{G}((H\cap L)K)|$ is a $p$-number.
The authors in \cite{GSY} showed that every $CAP$-subgroup and every $S$-quasinormally embedded
subgroup of $G$ are both a generalized $CAP$-subgroup of $G$, and the converse is not true. In connection with this, A. N. Skiba proposed the following question in Seminar of USTC, 2014:\par

\begin{Question}\label{que} {\bf (see also \cite[Chap. 1, Problem 6.14]{G3}).}  To study the structure of finite groups when the condition of every chief factor in the generalized $CAP$-subgroup is replaced by every chief factor in some chief series.
\end{Question}

The main objective of the paper is to give an answer to Question \ref{que}. For this purpose, we now introduce the following concept:\par

\begin{Definition}  A $p$-subgroup $H$ of $G$ is said to satisfy partial $S$-$\Pi$-property in $G$ if $G$ has a chief series $\Gamma_{G}: 1=G_{0}<G_{1}<\cdots<G_{n}=G$ such that for every $G$-chief factor $G_{i}/G_{i-1}$ $(1\leqslant i\leqslant n)$ of $\Gamma_{G}$, either $(H\cap G_{i})G_{i-1}/G_{i-1}$ is a Sylow $p$-subgroup of $G_{i}/G_{i-1}$ or $|G/G_{i-1}: N_{G/G_{i-1}}((H\cap G_{i})G_{i-1}/G_{i-1})|$ is a $p$-number.
\end{Definition}

It is clear that a $p$-subgroup $H$ of $G$ satisfy partial $S$-$\Pi$-property in $G$ if $H$ is a generalized $CAP$-subgroup of $G$. But the next example illustrates that the converse is not true.\par

\begin{Example}
 {\em Let $L_1=\langle a,b \,|\,a^5=b^5=1,ab=ba \rangle$ and $L_2=\langle a',b' \rangle$ be a copy of $L_1$. Let $\alpha$ be an automorphism of $L_1$ of order 3 satisfying that $a^\alpha=b$, $b^\alpha=a^{-1}b^{-1}$. Put $G=(L_1\times L_2)\rtimes \langle \alpha \rangle$ and $H=\langle a \rangle\times \langle a' \rangle$. Then $G$ has a minimal normal subgroup $N$ such that $H\cap N=1$. Note that $\mathit{\Gamma}_G:1<N<HN<G$ is a chief series of $G$. Then $H$ satisfies partial $S$-$\Pi$-property in $G$. But since $|G:N_G(H\cap L_1)|=|G:N_G(\langle a \rangle)|=3$, $H$ is not a generalized $CAP$-subgroup of $G$.\par}\end{Example}

Note also that, X. Chen and W. Guo in \cite{CX2} introduced the concept of partial $\Pi$-property: a subgroup $H$ of $G$ satisfies partial $\Pi$-property in $G$ if there exists a chief series $\Gamma_{G}: 1=G_{0}<G_{1}<\cdots<G_{n}=G$ of $G$ such that for every $G$-chief factor $G_{i}/G_{i-1}$ $(1\leqslant i\leqslant n)$ of $\Gamma_{G}$, $|G/G_{i-1}: N_{G/G_{i-1}}((H\cap G_{i})G_{i-1}/G_{i-1})|$ is a $\pi((H\cap G_{i})G_{i-1}/G_{i-1})$-number. It is easy to see that if a $p$-subgroup $H$ of $G$ satisfies partial $\Pi$-property in $G$, then $H$ satisfies partial $S$-$\Pi$-property in $G$. However, the converse does not hold in general.

\begin{Example}\label{e2}
 {\em Let $G=A_{5}$ and $H$ be a Sylow 5-subgroup of $A_{5}$, where $A_{5}$ is an alternative group of degree 5. Then it is easy to see that $H$ satisfies partial $S$-$\Pi$-property in $G$. However, since $|G: N_{G}(H)|$ is not a $5$-number, we have that $H$ does not satisfy partial $\Pi$-property in $G$.}\end{Example}

Let $\mathfrak{F}$ be a formation. The $\mathfrak{F}$-residual of $G$, denoted by $G^{\mathfrak{F}}$, is the smallest normal subgroup of $G$ with quotient in $\mathfrak{F}$. A $G$-chief factor $L/K$ is said to be $\mathfrak{F}$-central in $G$ if $L/K\rtimes G/C_{G}(L/K)\in \mathfrak{F}$. A normal subgroup $N$ of $G$ is called $\mathfrak{F}$-hypercentral in $G$ if either $N=1$ or every $G$-chief factor below $N$ is $\mathfrak{F}$-central in $G$. Let $Z_\mathfrak{F}(G)$ denote the $\mathfrak{F}$-hypercentre of $G$, that is, the product of all $\mathfrak{F}$-hypercentral normal subgroups of $G$. Moreover, the generalized Fitting subgroup $F^{\ast}(G)$ (resp. the generalized $p$-Fitting subgroup $F^{\ast}_p(G)$) of $G$ is quasinilpotent radical (resp. $p$-quasinilpotent radical) of $G$ (for details, see \cite [Chap. X]{HB1} and \cite{AE}). We denote the Fitting subgroup and the $p$-Fitting subgroup of $G$ by $F(G)$ and $F_p(G)$, respectively.\par

In this paper, we arrive at the following main results.

\begin{Theorem}\label{zh} Let $E$ and $X$ be normal subgroups of $G$ such that $F^{\ast}(E)\leq X\leq E$. Suppose that for any non-cyclic Sylow subgroup $P$ of $X$, every maximal subgroup of $P$ satisfies partial $S$-$\Pi$-property in $G$, or every cyclic subgroup of $P$ of prime order or order $4$ $($when $P$ is a non-abelian $2$-group$)$ satisfies partial $S$-$\Pi$-property in $G$. Then $E\leq Z_{\mathfrak{U}}(G)$.\end{Theorem}

\begin{Theorem}\label{pc} Let $E$ and $X$ be $p$-soluble normal subgroups of $G$ such that $F_p(E)\leq X\leq E$.
Suppose that $X$ has a Sylow $p$-subgroup $P$ such that every maximal subgroup of $P$ satisfies partial $S$-$\Pi$-property in $G$, or every cyclic subgroup of $P$ of prime order or order $4$ $($when $P$ is a non-abelian $2$-group$)$ satisfies partial $S$-$\Pi$-property in $G$. Then $E/O_{p'}(E)\leq Z_{\mathfrak{U}}(G/O_{p'}(E))$. \end{Theorem}

All unexplained notation and terminology are standard, as in \cite {KT,G2,Bal}.\par

\section{Preliminaries}

Firstly, we present some basic properties of partial $S$-$\Pi$-property as follows.

\begin{Lemma}\label{ma1} Suppose that a $p$-subgroup $H$ of $G$ satisfies partial $S$-$\Pi$-property in $G$ and $N\unlhd G$.\smallskip

$(1)$ If $H\leq N$, then  $H$ satisfies partial $S$-$\Pi$-property in $N$.\smallskip

$(2)$ If either $N\leq H$ or $(p,|N|)=1$, then $HN/N$ satisfies partial $S$-$\Pi$-property in $G/N$.\smallskip

$(3)$ If every maximal subgroup of a Sylow $p$-subgroup $P$ of $G$ satisfies partial $S$-$\Pi$-property in $G$, then every maximal subgroup of $PN/N$ also satisfies partial $S$-$\Pi$-property in $G/N$.
\end{Lemma}

\begin{proof}
By the hypothesis, we may assume that $G$ has a chief series $\Gamma_{G}: 1=G_{0}<G_{1}<\cdots<G_{n}=G$
such that for every $G$-chief factor $G_{i}/G_{i-1}$ $(1\leqslant i\leqslant n)$ of $\Gamma_{G}$, either $(H\cap G_{i})G_{i-1}/G_{i-1}$ is a Sylow $p$-subgroup of $G_{i}/G_{i-1}$ or $|G/G_{i-1}: N_{G/G_{i-1}}((H\cap G_{i})G_{i-1}/G_{i-1})|$ is a $p$-number.\smallskip

(1) Obviously, $\Gamma_{N}: 1=G_{0}\cap N\leq G_{1}\cap N\leq \cdots\leq G_{n}\cap N=N$ is a normal series of $N$. Let $L/K$ be an $N$-chief factor such that $G_{i-1}\cap N\leq K\leq L\leq G_{i}\cap N$ $(1\leqslant i\leqslant n)$. If $(H\cap G_{i})G_{i-1}/G_{i-1}$ is a Sylow $p$-subgroup of $G_{i}/G_{i-1}$, then $(H\cap G_{i})G_{i-1}/G_{i-1}$ is a Sylow $p$-subgroup of $(G_{i}\cap N)G_{i-1}/G_{i-1}$. We can deduce that $(H\cap G_{i})(G_{i-1}\cap N)/(G_{i-1}\cap N)$ is a Sylow $p$-subgroup of $(G_{i}\cap N)/(G_{i-1}\cap N)$. Hence $(H\cap L)K/K$ is a Sylow $p$-subgroup of $L/K$.
Now assume that $|G: N_{G}((H\cap G_{i})G_{i-1})|$ is a $p$-number. Then $|N: N_{N}((H\cap G_{i})(G_{i-1}\cap N))|$ is a $p$-number because $N_{G}((H\cap G_{i})G_{i-1})\leq N_{G}((H\cap G_{i})G_{i-1}\cap N)$. It is easy to see that $N_{N}((H\cap G_{i})(G_{i-1}\cap N))\leq N_N((H\cap L)K)$, and so $|N: N_{N}((H\cap L)K)|$ is a $p$-number.
Hence $H$ satisfies partial $S$-$\Pi$-property in $N$.\smallskip

(2) Note that if either $N\leq H$ or $(p,|N|)=1$, then $HN\cap XN=(H\cap X)N$ for every normal subgroup $X$ of $G$. Now consider the normal series $\Gamma_{G/N}: 1=G_{0}N/N\leq G_{1}N/N\leq \cdots\leq G_{n}N/N=G/N$ of $G/N$. For every normal section $G_iN/G_{i-1}N$, we have that $(HN\cap G_iN)G_{i-1}N=HG_{i-1}N\cap G_{i}N=(H\cap G_i)G_{i-1}N$. If $(H\cap G_{i})G_{i-1}/G_{i-1}$ is a Sylow $p$-subgroup of $G_{i}/G_{i-1}$, then $(HN\cap G_iN)G_{i-1}N/G_{i-1}N=(H\cap G_i)G_{i-1}N/G_{i-1}N$ is a Sylow $p$-subgroup of $G_{i}N/G_{i-1}N$. Now assume that $|G: N_{G}((H\cap G_{i})G_{i-1})|$ is a $p$-number. Then $|G: N_{G}((HN\cap G_{i}N)G_{i-1}N)|=|G: N_{G}((H\cap G_i)G_{i-1}N)|$ is a $p$-number. Therefore, $HN/N$ satisfies partial $S$-$\Pi$-property in $G/N$.\smallskip

(3) Let $T/N$ be any maximal subgroup of $PN/N$. Then there exists a maximal subgroup $P_{1}$ of $P$ such that $T=P_{1}N$ and $P_{1}\cap N=P\cap N$. It is easy to derive that $P_{1}N\cap XN=(P_{1}\cap X)N$ for any normal subgroup $X$ of $G$. With a similar argument as (2), we have that $T/N$ satisfies partial $S$-$\Pi$-property in $G/N$.\end{proof}

Let $P$ be a $p$-group. If $P$ is not a non-abelian $2$-group, then we use $\Omega(P)$ to denote the subgroup $\Omega_{1}(P)$. Otherwise, $\Omega(P)=\Omega_{2}(P)$.

\begin{Lemma}\label{zu} \textup{\cite[Lemma 2.12]{CX3}} Let $P$ be a normal $p$-subgroup of $G$ and $C$ a Thompson critical subgroup of $P$ \textup{(see \cite[\textit{p}. 186]{GO})}. If $P/\Phi(P)\leq Z_\mathfrak{U}(G/\Phi(P))$ or $C\leq Z_\mathfrak{U}(G)$ or $\Omega(P)\leq Z_\mathfrak{U}(G)$, then $P\leq Z_\mathfrak{U}(G)$.
\end{Lemma}

The next lemma is evident.

\begin{Lemma}\label{pm} Let $p$ be a prime divisor of $|G|$ with $(|G|,p-1)=1$ and $N$ a normal subgroup of $G$ such that $|N|_p\leq p$. If $G/N$ is $p$-nilpotent, then $G$ is $p$-nilpotent.\par
\end{Lemma}

\begin{Lemma} \label{3} \textup{\cite [Lemma 2.11]{CX3}} Let $P$ be a $p$-group of nilpotent class at most $2$. Suppose that the exponent of $P/Z(P)$ divides $p$.\smallskip

$(1)$ If $p>2$, then the exponent of $\Omega(P)$ is $p$.\smallskip

$(2)$ If $P$ is a non-abelian $2$-group, then the exponent of $\Omega(P)$ is $4$.\end{Lemma}

\begin{Lemma}\label{FZ}  \textup{\cite [Theorem B]{AN2}} Let $\mathfrak{F}$ be any formation and $E$ a normal subgroup of $G$. If $F^{\ast}(E)\leq Z_{\mathfrak{F}}(G)$, then $E\leq Z_{\mathfrak{F}}(G)$.\end{Lemma}

\begin{Lemma}\label{BE} \textup{\cite[Theorem 2.1.6]{ABl}} Let $G$ be a $p$-supersoluble group. Then the
derived subgroup $G'$ of $G$ is $p$-nilpotent. In particular, if $O_{p'}(G) = 1$, then $G$ has a
unique Sylow $p$-subgroup.\end{Lemma}

\section{Proof of Main Results}

The following propositions are the main steps of the proof of Theorems \ref{zh} and \ref{pc}.\par

\begin{Proposition}\label{1may} Let $P$ be a normal $p$-subgroup of $G$. If every maximal subgroup of $P$ satisfies partial $S$-$\Pi$-property in $G$, then $P\leq Z_{\mathfrak{U}}(G)$.\end{Proposition}

\begin{proof} Suppose that this proposition is false, and let $(G,P)$ be a counterexample for which $|G|+|P|$ is  minimal. Then:\smallskip

(1) \textit{There is a unique minimal normal subgroup $N$ of $G$ contained in $P$, $P/N\leq Z_{\mathfrak{U}}(G/N)$ and $|N|>p$.}
\smallskip

Let $N$ be any minimal normal subgroup of $G$ contained in $P$. Then clearly, $(G/N, P/N)$ satisfies the hypothesis by Lemma \ref{ma1}(2), and so the choice of $(G,P)$ yields that $P/N\leq Z_{\mathfrak{U}}(G/N)$. If $|N|=p$, then $P\leq Z_{\mathfrak{U}}(G)$, which is impossible. Hence $|N|>p$. Now suppose that $G$ has a minimal normal subgroup $R$ contained in $P$ such that $N\neq R$. With a similar discussion as above, we obtain that $P/R\leq Z_{\mathfrak{U}}(G/R)$. It follows that $NR/R\leq Z_{\mathfrak{U}}(G/R)$, and so $|N|=p$, a contradiction.\smallskip

(2) \textit{$\Phi(P)=1$, and so $P$ is elementary abelian.}\smallskip

If $\Phi(P)\neq 1$, then by (1), $N\leq \Phi(P)$. This induces that $P/\Phi(P)\leq Z_{\mathfrak{U}}(G/\Phi(P))$ because $P/N\leq Z_{\mathfrak{U}}(G/N)$, and so $P\leq Z_{\mathfrak{U}}(G)$ by Lemma \ref{zu}. This contradiction shows that $\Phi(P)=1$, and so $P$ is elementary abelian.\par\smallskip

(3) \textit{The final contradiction.}\smallskip

Let $N_{1}$ be a maximal subgroup of $N$ such that $N_{1}$ is normal in some Sylow $p$-subgroup of $G$, say $G_{p}$. Then $P_{1}=N_{1}S$ is a maximal subgroup of $P$, where $S$ is a complement of $N$ in $P$. By the hypothesis, $G$ has a chief series $\Gamma_{G}: 1=G_{0}<G_{1}<\cdots<G_{n}=G$ such that for every $G$-chief factor $G_{i}/G_{i-1}$ $(1\leqslant i\leqslant n)$ of $\Gamma_{G}$, either $(P_1\cap G_{i})G_{i-1}/G_{i-1}$ is a Sylow $p$-subgroup of $G_{i}/G_{i-1}$ or $|G/G_{i-1}: N_{G/G_{i-1}}((P_1\cap G_{i})G_{i-1}/G_{i-1})|$ is a $p$-number. Obviously, there exists an integer $k$ $(1\leqslant k\leqslant n)$ such that $G_{k}=G_{k-1}\times N$. It follows from (1) that $P\cap G_{k-1}=1$. If $(P_{1}\cap G_{k})G_{k-1}/G_{k-1}$ is a Sylow $p$-subgroup of $G_{k}/G_{k-1}$, then $G_{k}=(P_{1}\cap G_{k})G_{k-1}$, and thus $N\leq G_{k}\leq P_{1}G_{k-1}$. This implies that $N\leq P_{1}G_{k-1}\cap P\leq P_{1}$, a contradiction. Now assume that $|G: N_{G}((P_{1}\cap G_{k})G_{k-1})|$ is a $p$-number. Then $|G: N_{G}(P_{1}G_{k-1}\cap N))|$ is a $p$-number. Since $N_1\leq P_{1}G_{k-1}\cap N<N$, we have that $N_{1}=P_{1}G_{k-1}\cap N$. Hence $N_{1}\unlhd G$ because $N_{1}\unlhd G_{p}$, and so $N_1=1$, which is also a contradiction. The proof is thus completed.\end{proof}

\begin{Proposition}\label{1ma} Let  $E$ be a normal subgroup of $G$ and $p$ a prime divisor of $|E|$ with $(|E|, p-1)=1$. Suppose that $E$ has a Sylow $p$-subgroup $P$ such that every maximal subgroup of $P$ satisfies partial $S$-$\Pi$-property in $G$. Then $E$ is $p$-nilpotent.\end{Proposition}

\begin{proof}
Suppose that this proposition is false, and let $(G,E)$ be a counterexample for which $|G|+|E|$ is  minimal. Now we proceed the proof via the following steps.\smallskip

(1) \textit{$O_{p'}(E)=1$.}\smallskip

If $O_{p'}(E)\neq1$, then by Lemma \ref{ma1}(2), $(G/O_{p'}(E),E/O_{p'}(E))$ satisfies the hypothesis. The choice of $(G,E)$ yields that $E/O_{p'}(E)$ is $p$-nilpotent, and so $E$ is $p$-nilpotent, a contradiction. Hence $O_{p'}(E)=1$.\smallskip

(2) \textit{$E=G$.}\smallskip

Assume that $E<G$. Then by Lemma \ref{ma1}(1), the hypothesis holds for $(E,E)$. By the choice of the $(G,E)$, $E$ is $p$-nilpotent. This contradiction shows that $E=G$.

\smallskip
(3) \textit{$G$ has a unique minimal normal subgroup $N$, $G/N$ is $p$-nilpotent and $\Phi(G)=1$.}\smallskip

Let $N$ be any minimal normal subgroup of $G$. Then by Lemma \ref{ma1}(3), the hypothesis still holds for $(G/N, G/N)$. The choice of $(G,E)$ yields that $G/N$ is $p$-nilpotent. Hence $N$ is the unique minimal normal subgroup of $G$, and it is easy to see that $\Phi(G)=1$.\smallskip

(4) \textit{The final contradiction.}\smallskip

If $P\cap N\leq \Phi(P)$, then by \cite [Chap. IV, Satz 4.7]{HB}, $N$ is $p$-nilpotent. Hence by (1), $N$ is a $p$-group, and so $N\leq \Phi(G)=1$, which is impossible. Thus $P\cap N\nleq \Phi(P)$. Then $P$ has a maximal subgroup $P_{1}$ such that $P=P_1(P\cap N)$. By the hypothesis and (3), $G$ has a chief series $\Gamma_{G}: 1=G_{0}<G_{1}=N<\cdots<G_{n}=G$ such that for every $G$-chief factor $G_{i}/G_{i-1}$ $(1\leqslant i\leqslant n)$ of $\Gamma_{G}$, either $(P_1\cap G_{i})G_{i-1}/G_{i-1}$ is a Sylow $p$-subgroup of $G_{i}/G_{i-1}$ or $|G/G_{i-1}: N_{G/G_{i-1}}((P_1\cap G_{i})G_{i-1}/G_{i-1})|$ is a $p$-number. Evidently, $P_{1}\cap N$ is not a Sylow $p$-subgroup of $N$. Therefore, $|G:N_{G}(P_{1}\cap N)|$ is a $p$-number. Since $P\leq N_{G}(P_{1}\cap N)$, we have that $P_{1}\cap N\unlhd G$. It follow that $P_{1}\cap N=1$, and thereby $|N|_p=p$. Then by  (3) and Lemma \ref{pm}, $G$ is $p$-nilpotent, a contradiction. This completes the proof.
\end{proof}

\begin{Proposition}\label{1miy} Let $P$ be a normal $p$-subgroup of $G$. If every cyclic subgroup of $P$ of prime order or order $4$ $($when $P$ is a non-abelian $2$-group$)$ satisfies partial $S$-$\Pi$-property in $G$, then $P\leq Z_{\mathfrak{U}}(G)$.\end{Proposition}

\begin{proof} Suppose that this proposition is false, and let $(G,P)$ be a counterexample for which $|G|+|P|$ is  minimal. Then:\smallskip

(1) \textit{$G$ has a unique normal subgroup $N$ such that $P/N$ is a $G$-chief factor, $N\leq Z_{\mathfrak{U}}(G)$ and $|P/N|>p$.}\smallskip

Let $P/N$ be a $G$-chief factor. Then $(G, N)$ satisfies the hypothesis, and the choice of $(G, P)$ implies that $N\leq Z_{\mathfrak{U}}(G)$. If $|P/N|=p$, then $P/N\leq Z_{\mathfrak{U}}(G/N)$, and so $P\leq Z_{\mathfrak{U}}(G)$, which is impossible.
Thus $|P/N|>p$. Now assume that $P/R$ is a $G$-chief factor with $N\neq R$. With a similar argument as above, we have that $R\leq Z_{\mathfrak{U}}(G)$. This yields that $P=NR\leq Z_{\mathfrak{U}}(G)$, a contradiction occurs. Therefore, $N$ is the unique normal subgroup of $G$ such that $P/N$ is a $G$-chief factor.\smallskip

(2) \textit{The exponent of $P$ is $p$ or $4$ $($when $P$ is a non-abelian $2$-group$)$.}\smallskip

Let $D$ be a Thompson critical subgroup of $P$. Then the nilpotent class of $D$ is at most $2$ and $D/Z(D)$ is elementary abelian by \cite[Chap. 5, Theorem 3.11]{GO}. If $\Omega(D)<P$, then $\Omega(D)\leq N\leq Z_{\mathfrak{U}}(G)$ by (1). It follows from Lemma \ref{zu} that $P\leq Z_{\mathfrak{U}}(G)$, against supposition. Thus $P=D=\Omega(D)$. Then by Lemma \ref{3}, the exponent of $P$ is $p$ or $4$ (when $P$ is a non-abelian $2$-group).

\smallskip
(3) \textit{The final contradiction.}\smallskip

Let $G_p$ be a Sylow $p$-subgroup of $G$. Since $P/N\cap Z(G_{p}/N)>1$, there exists a subgroup $T/N$ of $P/N\cap Z(G_{p}/N)$ of order $p$. Let $x\in T \backslash N$ and $H=\langle x\rangle$. Then $T=HN$ and $|H|=p$ or $4$ (when $P$ is a non-abelian $2$-group) by (2). By the hypothesis, $G$ has a chief series $\Gamma_{G}: 1=G_{0}<G_{1}<\cdots<G_{n}=G$ such that for every $G$-chief factor $G_{i}/G_{i-1}$ $(1\leqslant i\leqslant n)$ of $\Gamma_{G}$, either $(H\cap G_{i})G_{i-1}/G_{i-1}$ is a Sylow $p$-subgroup of $G_{i}/G_{i-1}$ or $|G/G_{i-1}: N_{G/G_{i-1}}((H\cap G_{i})G_{i-1}/G_{i-1})|$ is a $p$-number. Clearly, there exists an integer $k$ $(1\leqslant k\leqslant n)$ such that $P\leq G_{k}$ and $P\nleq G_{k-1}$.
Since $N$ is the unique normal subgroup of $G$ such that $P/N$ is a $G$-chief factor by (1), we have that $P\cap G_{k-1}\leq N$. If $G_{k}=NG_{k-1}$, then $P=N(P\cap G_{k-1})=N$. This contradiction forces that $N\leq G_{k-1}$, and so $P\cap G_{k-1}=N$.

Firstly suppose that $HG_{k-1}/G_{k-1}$ is a Sylow $p$-subgroup of $G_{k}/G_{k-1}$. Then $P\leq HG_{k-1}$. This implies that $P=H(P\cap G_{k-1})=HN=T$, and so $|P/N|=|T/N|=p$, which contradicts (1).
Now assume that $|G: N_{G}(HG_{k-1})|$ is a $p$-number, and so $|G: N_{G}(T)|$ is a $p$-number. Since $G_{p}\leq N_{G}(T)$, we have that $T\unlhd G$. It follows from (1) that $P=T$ because $H\neq N$, a contradiction also occurs. This ends the proof. \end{proof}

\begin{Proposition}\label{1mi} Let $E$ be a normal subgroup of $G$ and $p$ a prime divisor of $|E|$ with $(|E|, p-1)=1$. Suppose that $E$ has a Sylow $p$-subgroup $P$ such that every cyclic subgroup of $P$ of order $p$ or $4$ $($when $P$ is a non-abelian $2$-group$)$ satisfies partial $S$-$\Pi$-property in $G$. Then $E$ is $p$-nilpotent.\end{Proposition}

\begin{proof}
Suppose that this proposition is false, and let $(G,E)$ be a counterexample for which $|G|+|E|$ is  minimal. Now we proceed the proof via the following steps.\smallskip

(1) \textit{$O_{p'}(G)=1$ and $E=G$.}\smallskip

With a similar argument as in steps (1) and (2) of the proof of Proposition \ref{1ma}, we have that $O_{p'}(G)=1$ and $E=G$.

\smallskip
(2) \textit{$Z(G)$ is the unique normal subgroup of $G$ such that $G/Z(G)$ is a $G$-chief factor,  $Z(G)=Z_{\infty}(G)=O_{p}(G)$ and $G^{\mathfrak{N}}=G$.}\smallskip

Let $G/K$ be a $G$-chief factor. Then by Lemma \ref{ma1}(1), $(K,K)$ satisfies the hypothesis. The choice of $(G,E)$ yields that $K$ is $p$-nilpotent, and so $K\leq F_p(G)$. Since $G$ is not $p$-nilpotent, $K=F_p(G)$. This shows that $F_p(G)$ is the unique normal subgroup of $G$ such that $G/F_p(G)$ is a $G$-chief factor. By (1) and Proposition \ref{1miy}, $F_p(G)=O_{p}(G)\leq Z_{\mathfrak{U}}(G)$. As $(|G|, p-1)=1$, we have that $O_{p}(G)\leq Z_{\infty}(G)$, and thereby $Z_{\infty}(G)=O_{p}(G)$ because $Z_{\infty}(G)\leq F_p(G)=O_{p}(G)$. If $G^{\mathfrak{N}}<G$, then $G^{\mathfrak{N}}\leq Z_{\infty}(G)$. This implies that $G$ is $p$-nilpotent, a contradiction. Thus $G^{\mathfrak{N}}=G$. Then by \cite [Chap. IV, Theorem 6.10]{KT}, $Z_{\infty}(G)\leq Z(G)$, and so $Z_{\infty}(G)=Z(G)$.\smallskip

(3) \textit{$P$ is non-abelian.}\smallskip

If $P$ is abelian, then by (1) and \cite[Chap. VI, Satz 14.3]{HB}, $G'\cap Z(G)=O_{p'}(G)=1$. Since $G'=G$ by (2), we have that $Z(G)=1$, and so $G$ is a simple group. Let $x$ be an element of $G$ of order $p$. Then by the hypothesis, either $\langle x\rangle$ is a Sylow $p$-subgroup of $G$ or $|G: N_{G}(\langle x\rangle)|$ is a $p$-number. In the former case, $G$ is $p$-nilpotent by Lemma \ref{pm}, a contradiction. In the latter case, $\langle x\rangle\leq O_p(G)=1$ by (2), also a contradiction. Thus $P$ is non-abelian.\smallskip

(4) \textit{The final contradiction.}\smallskip

Suppose that all cyclic subgroups of $P$ of order $p$ and $4$ are contained in $Z(G)$, then $G$ is $p$-nilpotent by \cite[Chap. IV, Satz 5.5]{HB}. Hence $G$ has an element $x$ of order $p$ or $4$ such that $x\notin Z(G)$.
Then by (2), (3) and the hypothesis, $G$ has a chief series $\Gamma_{G}: 1=G_{0}<G_{1}<\cdots<G_{n-1}=Z(G)<G_{n}=G$ such that for every $G$-chief factor $G_{i}/G_{i-1}$ $(1\leqslant i\leqslant n)$ of $\Gamma_{G}$, either $(\langle x\rangle\cap G_{i})G_{i-1}/G_{i-1}$ is a Sylow $p$-subgroup of $G_{i}/G_{i-1}$ or $|G/G_{i-1}: N_{G/G_{i-1}}((\langle x\rangle\cap G_{i})G_{i-1}/G_{i-1})|$ is a $p$-number.
If $\langle x\rangle Z(G)/Z(G)$ is a Sylow $p$-subgroup of $G/Z(G)$, then $P=\langle x\rangle Z(G)$, and so $P$ is abelian, which contradicts (3). Now assume that $|G: N_{G}(\langle x\rangle Z(G))|$ is a $p$-number.
Then $\langle x\rangle Z(G)\leq O_p(G)=Z(G)$ by (2). This contradiction completes the proof. \end{proof}

Now we are ready to prove Theorem \ref{zh}.\par

\begin{proof}[{\bf Proof of Theorem \ref{zh}.}] Let $p$ be the smallest prime divisor of $|X|$ and $X_p$ a Sylow $p$-subgroup of $X$. If $X_p$ is cyclic, then $X$ is $p$-nilpotent by \cite[10.1.9]{Rob}. Now assume that $X_p$ is not cyclic. Then by Lemma \ref{ma1}(1), Propositions \ref{1ma} and \ref{1mi}, $X$ is also $p$-nilpotent. Let $X_{p'}$ be the normal $p$-complement of $X$. Then $X_{p'}\unlhd G$. If $X_p$ is cyclic, then clearly, $X/X_{p'}\leq Z_{\mathfrak{U}}(G/X_{p'})$. Now assume that $X_p$ is not cyclic. Then it is easy to see that $(G/X_{p'}, X/X_{p'})$ satisfies the hypothesis of Proposition \ref{1may} or Proposition \ref{1miy} by Lemma \ref{ma1}(2). Hence we also have that $X/X_{p'}\leq Z_{\mathfrak{U}}(G/X_{p'})$.

Let $q$ be the second smallest prime divisor of $|X|$. By arguing similarly as above, we obtain that $X_{p'}$ is $q$-nilpotent and $X_{p'}/X_{\{p,q\}'}\leq Z_{\mathfrak{U}}(G/X_{\{p,q\}'})$, where $X_{\{p,q\}'}$ is the normal $q$-complement of $X_{p'}$. The rest can be deduced by analogy. Hence we can obtain that $X\leq Z_{\mathfrak{U}}(G)$. Then by Lemma \ref{FZ}, $E\leq Z_{\mathfrak{U}}(G)$. The theorem is thus proved.\end{proof}

In order to prove Theorem \ref{pc}, we need the following proposition.\par

\begin{Proposition}\label{po} Let $E$ be a $p$-soluble normal subgroup of $G$. Suppose that $E$ has a Sylow $p$-subgroup $P$ such that every maximal subgroup of $P$ satisfies partial $S$-$\Pi$-property in $G$, or every cyclic subgroup of $P$ of prime order or order $4$ $($when $P$ is a non-abelian $2$-group$)$ satisfies partial $S$-$\Pi$-property in $G$. Then $E/O_{p'}(E)\leq Z_{\mathfrak{U}}(G/O_{p'}(E))$.\end{Proposition}

\begin{proof} Suppose that this proposition is false, and let $(G,E)$ be a counterexample for which $|G|+|E|$ is  minimal. With a similar discussion as in step (1) of the proof of Proposition \ref{1ma}, we have that $O_{p'}(E)=1$. If $E$ is $p$-supersoluble, then $E'\leq F(E)=P$ by Lemma \ref{BE}. Hence $E$ is soluble, and so $F^*(E)=F(E)=P$ by \cite[Chap. X, Corollary 13.7(d)]{HB1}. Note that by Propositions \ref{1may} and \ref{1miy}, we have that $P\leq Z_{\mathfrak{U}}(G)$. It follows from Lemma \ref{FZ} that $E\leq Z_{\mathfrak{U}}(G)$, which is impossible. Thus $E$ is not $p$-supersoluble. If $E<G$, then since $(E,E)$ satisfies the hypothesis by Lemma \ref{ma1}(1), $E$ is $p$-supersoluble by the choice of $(G,E)$. This contradiction implies that $E=G$ and $G$ is not $p$-supersoluble.

Firstly suppose that every maximal subgroup of $P$ satisfies partial $S$-$\Pi$-property in $G$.
Let $N$ be a minimal normal subgroup of $G$. Since $G$ is $p$-soluble and $O_{p'}(G)=1$, we see that $N\leq O_p(G)$. By Lemma \ref{ma1}(2), the hypothesis holds for $(G/N,G/N)$, and so the choice of $(G,E)$ implies that $G/N$ is $p$-supersoluble. Then it is easy to see that $N$ is the unique minimal normal subgroup of $G$ and $N\nleq\Phi(G)$. Hence $P$ has a maximal subgroup $P_{1}$ such that $P=P_{1}N$. Then by the hypothesis, $G$ has a chief series $\Gamma_{G}: 1=G_{0}<G_{1}=N<\cdots<G_{n}=G$ such that for every $G$-chief factor $G_{i}/G_{i-1}$ $(1\leqslant i\leqslant n)$ of $\Gamma_{G}$, either $(P_1\cap G_{i})G_{i-1}/G_{i-1}$ is a Sylow $p$-subgroup of $G_{i}/G_{i-1}$ or $|G/G_{i-1}: N_{G/G_{i-1}}((P_1\cap G_{i})G_{i-1}/G_{i-1})|$ is a $p$-number. Since $N\nleq P_1$, we have that $|G: N_{G}(P_1\cap N)|$ is a $p$-number. Hence $P_1\cap N\unlhd G$. It follows that $P_1\cap N=1$, and so $|N|=p$. Thus $G$ is $p$-supersoluble, a contradiction.

Now assume that every cyclic subgroup of $P$ of prime order or order $4$ $($when $P$ is a non-abelian $2$-group$)$ satisfies partial $S$-$\Pi$-property in $G$. Let $G/K$ be a $G$-chief factor. Then $G/K$ is $p$-supersoluble because $G/K$ is a $p$-soluble simple group, and $(K,K)$ satisfies the hypothesis by Lemma \ref{ma1}(1). By the choice of $(G,E)$, $K$ is $p$-supersoluble. Since $O_{p'}(K)\leq O_{p'}(G)=1$, $P\cap K\unlhd G$ by Lemma \ref{BE}. Then by Proposition \ref{1miy}, $P\cap K\leq Z_\frak{U}(G)$. As $G/(P\cap K)$ is $p$-supersoluble, we have that $G$ is $p$-supersoluble. The final contradiction completes the proof.
\end{proof}

\begin{proof}[{\bf Proof of Theorem \ref{pc}.}] Note that $O_{p'}(X)=O_{p'}(E)$. Then by Proposition \ref{po}, we have that $X/O_{p'}(E)\leq Z_{\mathfrak{U}}(G/O_{p'}(E))$, and so $F_p(E)/O_{p'}(E)\leq Z_{\mathfrak{U}}(G/O_{p'}(E))$.
It follows from \cite[Lemma 2.10]{AE} that $F^*(E/O_{p'}(E))=F_p^*(E/O_{p'}(E))=F_p(E)/O_{p'}(E)\leq Z_{\mathfrak{U}}(G/O_{p'}(E))$ because $E$ is $p$-soluble. Hence $E/O_{p'}(E)\leq Z_{\mathfrak{U}}(G/O_{p'}(E))$ by Lemma \ref{FZ}.
\end{proof}

\section{\bf Final Remarks}

In this section, we shall show that the concept of partial $S$-$\Pi$-property can be viewed as a generalization of many known embedding properties. Though some of them are generalized by the concept of partial $\Pi$-property, there are still some embedding properties can only be generalized by the concept of partial $S$-$\Pi$-property as the following proposition illustrates. Hence, as a consequence, a large number of results in former literature can follow directly from our main results.\par

\begin{Proposition}\label{C1} Let $H$ be a $p$-subgroup of $G$. Then $H$ satisfies partial $S$-$\Pi$-property in $G$ if one of the following holds:
\smallskip

$(1)$ $H$ is a generalized $CAP$-subgroup of $G$.
\smallskip

$(2)$ $H$ satisfies partial $\Pi$-property in $G$.
\smallskip

$(3)$ $H$ is $\Pi$-normal\textup{\cite{LB}} in $G$.
\smallskip

$(4)$ $H$ is $\mathfrak{U}_{c}$-normal\textup{\cite{AY}} in $G$.
\smallskip

$(5)$ $H$ is weakly $S$-permutable\textup{\cite{AN}} in $G$.
\smallskip

$(6)$ $H$ is weakly $S$-semipermutable\textup{\cite{LY1}} in $G$.
\smallskip

$(7)$ $H$ is weakly $SS$-permutable\textup{\cite{HE}} in $G$.
\smallskip

$(8)$ $H$ is weakly $\tau$-quasinormal\textup{\cite{VO}} in $G$.
\smallskip

$(9)$ $H$ is $SE$-quasinormal\textup{\cite{CX3}} in $G$.
\smallskip

$(10)$ $H$ is a partial CAP-subgroup \textup{(}or semi CAP-subgroup\textup{)}\textup{\cite{FY}} of $G$.
\smallskip

$(11)$ $H$ is $S$-embedded\textup{\cite{G1}} in $G$.
\smallskip

$(12)$ $H$ is $\mathfrak{U}$-quasinormal\textup{\cite{ML}} in $G$.
\smallskip

$(13)$ $H$ is $\mathfrak{U}_{s}$-quasinormal\textup{\cite{HJ}} in $G$.
\smallskip

$(14)$ $H$ is weakly $S$-embedded\textup{\cite{LJ}} in $G$.
\end{Proposition}

\begin{proof} Statements (1) and (2) hold by the definition, and statements (3)-(8) and (10)-(13) follow from \cite[Lemmas 7.2 and 7.3]{CX2}.\smallskip

(9) By the definition, $G$ has a subnormal subgroup $T$ such that $G=HT$ and $H\cap T\leq H_{seG}$, where $H_{seG}$ denotes the subgroup generated by all subgroups of $H$ which are $S$-quasinormally embedded in $G$. Then clearly, $O^p(G)\leq T$. Let $H_1,H_2,\cdots ,H_n$ be all subgroups of $H$ which are $S$-quasinormally embedded in $G$. Then there exist $S$-quasinormal subgroups $X_1,X_2,\cdots ,X_n$ of $G$ with $H_i$ is a Sylow $p$-subgroup of $X_i$ ($1\leq i\leq n$). If ${(X_i)}_G=1$ for all $1\leq i\leq n$, then by \cite[Theorems 1.2.14 and 1.2.17]{ABl}, $H_i$ is $S$-quasinormal in $G$ for all $1\leq i\leq n$, and so $H_{seG}=\langle H_1,H_2,\cdots ,H_n\rangle$ is $S$-quasinormal in $G$. Thus $H$ is weakly $S$-permutable in $G$. This implies that $H$ satisfies partial $S$-$\Pi$-property in $G$ by \cite[Lemma 7.3(3)]{CX2}. Hence, without loss of generality, we may assume that ${(X_1)}_G\neq 1$.

Suppose that $O^p(G)\cap {(X_1)}_G\neq 1$. Let $N$ be a minimal normal subgroup of $G$ contained in $O^p(G)\cap {(X_1)}_G$. It is easy to see that $HN/N$ is $SE$-quasinormal in $G/N$ by \cite[Lemma 2.6(3)]{CX3}. By induction, we have that $HN/N$ satisfies partial $S$-$\Pi$-property in $G/N$. Since $N\leq X_1$ and $H_1$ is a Sylow $p$-subgroup of $X_1$, $H_1\cap N$ is a Sylow $p$-subgroup of $N$, and so $H\cap N$ is a Sylow $p$-subgroup of $N$. This shows that $H$ satisfies partial $S$-$\Pi$-property in $G$. Now assume that $O^p(G)\cap {(X_1)}_G=1$. Let $R$ be a minimal normal subgroup of $G$ contained in ${(X_1)}_G$. Then $R\leq O_p(G)$. This implies that $R\leq H_1\leq H$. Since $H/R$ is $SE$-quasinormal in $G/R$ by \cite[Lemma 2.7(2)]{CX3}, we have that $H/R$ satisfies partial $S$-$\Pi$-property in $G/R$ by induction. Therefore, $H$ also satisfies partial $S$-$\Pi$-property in $G$.\smallskip

(14) By the definition, $G$ has a normal subgroup $T$ such that $HT$ is $S$-quasinormal in $G$ and $H\cap T\leq H_{seG}$, where $H_{seG}$ denotes the subgroup generated by all subgroups of $H$ which are $S$-quasinormally embedded in $G$. Let $H_1,H_2,\cdots ,H_n$ be all subgroups of $H$ which are $S$-quasinormally embedded in $G$. Then there exist $S$-quasinormal subgroups $X_1,X_2,\cdots ,X_n$ of $G$ with $H_i$ is a Sylow $p$-subgroup of $X_i$ ($1\leq i\leq n$). Without loss of generality, we assume that $X_i\leq HT$ ($1\leq i\leq n$). If ${(X_i)}_G=1$ for all $1\leq i\leq n$, then $H_{seG}$ is $S$-quasinormal in $G$. Thus $H$ is $S$-embedded in $G$. This yields that $H$ satisfies partial $S$-$\Pi$-property in $G$ by \cite[Lemma 7.2(2)]{CX2}. Hence we may assume that ${(X_1)}_G\neq 1$.

Suppose that $T\cap {(X_1)}_G\neq 1$. Then we can obtain that $H$ satisfies partial $S$-$\Pi$-property in $G$ by arguing similarly as in the proof of (9). Now assume that $T\cap {(X_1)}_G=1$. Let $R$ be a minimal normal subgroup of $G$ contained in ${(X_1)}_G$. Since $R\cap T=1$ and $R\leq HT$, we have that $R\leq H$. With a similar discussion as in the proof of (9), $H$ also satisfies partial $S$-$\Pi$-property in $G$.
\end{proof}

\end{document}